\documentclass[11pt]{amsart}
\usepackage{xcolor}
\usepackage{fancyhdr}
\usepackage{lastpage}
\usepackage[margin=1in]{geometry}
\usepackage{geometry}
\usepackage{hyperref}
\hypersetup{
    colorlinks=true,
    linkcolor=blue,
    citecolor=brown,
    filecolor=magenta,
    urlcolor=blue,
}
\usepackage{amsmath}
\usepackage{amsfonts}
\usepackage{amssymb}
\usepackage{amsthm}
\usepackage{setspace}
\usepackage{comment}
\usepackage{mathtools}
\usepackage{graphicx}
\usepackage[shortlabels]{enumitem}

\numberwithin{equation}{section}

\newtheorem{thm}{Theorem}[section]

\newtheorem{cor}[thm]{Corollary}
\newtheorem{lem}[thm]{Lemma}
\newtheorem{prop}[thm]{Proposition}

\theoremstyle{definition}

\newtheorem{defn}[thm]{Definition}

\newtheorem*{defn*}{Definition}

\newtheorem{ex}[thm]{Example}

\newtheorem*{rmk*}{Remark}

\newcommand{\Addresses}{{% additional braces for segregating \footnotesize
\bigskip
\footnotesize
Tianhang Lu$^\ast$: \texttt{lutianhang@stu.ouc.edu.cn}
\medskip
$\ast$: \textsc{School of Mathematical Sciences, Ocean University of China, Qingdao, China}\par\nopagebreak
}}

\title[]{Note about network decomposition by maximum flow}
\author{Tianhang Lu}

\begin{document}

\maketitle

\section{The Maximum Flow Problems and Minimum Cut Problems} \label{Sec:flow-cut}

Given a directed network $D=(V, E)$, where $s,t \in V$ represent the source and sink of the network respectively, and $c$ represents the non-negative capacity of each edge.
If a mapping $f: E \rightarrow \mathbb{R}^+$ satisfies the following conditions:
\begin{enumerate}
\item Capacity constraint: $f(e) \leq c(e)$, for all $e \in E$;
\item Flow conservation: $\sum_{(u,v)\in E} f(u,v) = \sum_{(v,u)\in E} f(v,u)$, for all $v \in V \setminus \{s, t\}$,
\end{enumerate}
then $f$ is referred to as a \textit{flow} on $D$.
The flow value of $f$, denoted by $|f|$, is defined as the amount of flow leaving the source $s$ minus the amount of flow returning to $s$:

$$
|f| = \sum_{(s, v) \in E} f(s, v) - \sum_{(v, s) \in E} f(v, s).
$$

The maximum flow problem seeks to find the flow with the maximum value in the network.
This problem can be formulated as the following linear program:

\begin{equation}\label{LP:max-flow}
    \begin{array}{cl}
    \text{max} & \sum\limits_{v\in V} f(s, v) \\[1mm]
    \text{subject to} & \sum\limits_{v\in V} f(u, v) - \sum\limits_{v\in V} f(v, u) = 0, ~u \in V \setminus \{s, t\}, \\[1mm]
    & 0 \leq f(e) \leq c(e),~ e \in E.
    \end{array}   
\end{equation}

Given a subset $S \subseteq V$ in network $D$, the set $\delta^+(S) = \{(u,v) \in E \mid u \in S, v \in V \setminus S\}$ is called a \textit{cut} in $D$.
Similarly, we define $\delta^-(S) = \delta^+(V \setminus S)$, representing the set of edges entering $S$.
If $s \in S$ and $t \notin S$, the cut $\delta^+(S)$ is called an $s-t$ cut.
The capacity of the cut $\delta^+(S)$, denoted $c(\delta^+(S))$, is defined as the sum of capacities of all edges in $\delta^+(S)$:

$$
c(\delta^+(S)) = \sum_{e \in \delta^+(S)} c(e).
$$

The minimum cut problem seeks to find the $s-t$ cut with the minimum capacity.
The celebrated Max-Flow Min-Cut Theorem, initially proven by Lester Ford and Delbert Fulkerson \cite{Ford1956maximal}, establishes the relationship between maximum flow and minimum cut:

\begin{thm}[Max-flow min-cut theorem \cite{Ford1956maximal}]
The value of the maximum $s-t$ flow in a network equals the capacity of the minimum $s-t$ cut.
\end{thm}

The goal of this paper is to establish a decomposition of the network based on the maximum flow problem.
The remainder of this paper is organized as follows.
In Section \ref{Sec:SCC}, we introduce the strict complementarity conditions from linear programming, and then use them to obtain a partition of the edges in the network.
Section \ref{Sec:net-dec} presents a decomposition of the network based on this edge partition, which characterizes all possible maximum flows in the network.
In Section \ref{Sec:dec-cut}, we study all the minimum cuts in the network and analyze their relationship with the network decomposition.
In Section \ref{Sec:dec-pot}, we use a potential function to provide an equivalent description of minimum cuts and analyze the distribution of the potential function in the network.

\section{Strict Complementarity Conditions} \label{Sec:SCC}

Consider the following linear program:

\begin{equation} \tag{P}
\begin{array}{cl} \label{LP:P}
\text{max} & c^T x \\[1mm]
\text{subject to} & A x \leq b, \\[1mm]
& x \geq 0,
\end{array}
\end{equation}

and its dual program:

\begin{equation} \tag{D}
\begin{array}{cl}  \label{LP:D}
\text{min} & b^T y \\[1mm]
\text{subject to} & A^T y \geq c, \\[1mm]
& y \geq 0,
\end{array}
\end{equation}

where $A \in \mathbb{R}^{m \times n}$, $c, x, s \in \mathbb{R}^n$, and $b \in \mathbb{R}^m$.
There exist strict complementarity conditions:

\begin{thm}[{Strict complementarity conditions, see e.g. \cite[Section~7.9]{Sch1987theory}}] \label{Thm:scc}
For any $i=1,\ldots,n$, either of the following holds:
\begin{enumerate}
    \item There exists an optimal solution $x$ of (\ref{LP:P}) such that $x_i > 0$;
    \item There exists an optimal solution $y$ of (\ref{LP:D}) such that $A_i^T y > c_i$,
\end{enumerate}
where $A_i$ is the $i$-th row of matrix $A$.
\end{thm}

Let $[n] = \{1, \ldots, n\}$ represent the index set of variables, which can be partitioned into two parts: $(N_1, N_2)$, where:

$$
\begin{aligned}
N_1 = \{i \in [n] \mid x_i > 0\}, \\
N_2 = \{i \in [n] \mid A^T_i y > 0\}.
\end{aligned}
$$

This partition is called the \textit{optimal partition} of the variables $x$.
Similarly, the index set $[m] = \{1, \ldots, m\}$ of the dual variables can be partitioned into two parts: $(M_1, M_2)$, where:

$$
\begin{aligned}
M_1 = \{j \in [m] \mid A_j x < b_j\}, \\
M_2 = \{j \in [m] \mid y_j > 0\}.
\end{aligned}
$$

For the maximum flow problem, the dual linear program to (\ref{LP:max-flow}) is:

\begin{equation} \label{LP:min-cut}
\begin{array}{cl}
\text{min} & \sum\limits_{e \in E} c(e) y(e) \\[1mm]
\text{subject to} & \pi(v) + y(s,v) \geq 1, ~(s,v) \in E, \\[1mm]
& \pi(v) - \pi(u) + y(u,v) \geq 0,~(u,v) \in E,~u \neq s,~v \neq t, \\[1mm]
& y(u,t) \geq \pi(u),~(u,t) \in E, \\
& y(e) \geq 0,~e \in E.
\end{array}
\end{equation}

By defining $\pi(s) = 1$ and $\pi(t) = 0$, this dual program can be rewritten as:

\begin{equation} \label{LP:re-min-cut}
\begin{array}{cl}
\text{min} & \sum\limits_{e \in E} c(e) y(e) \\[1mm]
\text{subject to} & \pi(v) - \pi(u) + y(u,v) \geq 0,~(u,v) \in E, \\[1mm]
& \pi(s) = 1, \\
& \pi(t) = 0, \\
& y(e) \geq 0,~e \in E.
\end{array}
\end{equation}

From the optimal partition of the variables $f$ and $y$, we can obtain a partition of the edges in $D$:

\begin{defn}  \label{Defn:edge}
Given a network $D$, an edge $e \in E$ is classified as follows:
\begin{enumerate}
    \item If $f(e) = c(e)$ for all maximum flows $f$, then $e$ is called an \textit{essential} edge.
Otherwise, it is called a \textit{dummy} edge;
    \item If there exists a maximum flow $f$ such that $0 < f(e) < c(e)$, then $e$ is called a \textit{dummy I} edge;
    \item If $f(e) = 0$ for all maximum flows $f$, then $e$ is called a \textit{dummy II} edge.
\end{enumerate}
\end{defn}

Let $A, C, R$ represent the sets of essential edges, dummy I edges, and dummy II edges, respectively.

\section{Decomposition of the Network by Maximum Flow} \label{Sec:net-dec}

Given that $f$ is the maximum flow in $D$, edges in $D$ that have the same direction as the flow $f$ are called the \textit{forward edges} of $f$, while those in the opposite direction are called \textit{backward edges}.
The residual network of $f$ is defined as the subgraph consisting of forward edges that are not saturated and backward edges that are not empty.
More formally, we define the residual network of $f$ as $D_f = (V, E_f)$, where $E_f = \{(u, v) \in V \times V \mid c(u, v) - f(u, v) > 0\}$.

\begin{defn} \label{Defn:SCC}
Given a network $D$, a maximal subgraph in which there is a directed path between every pair of vertices is called a \textit{strongly connected component} of $D$.
\end{defn}

Let $\mathcal{B}_f$ represent the set of vertex sets of strongly connected components in the residual network $D_f$.
For any $B \in \mathcal{B}_f$, the subgraph $D_f[B]$ induced by the vertices in $B$ is a strongly connected component.
We will prove that $\mathcal{B}_f$ is independent of the choice of $f$.

\begin{lem}  \label{Defn:B}
If $f_1$ and $f_2$ are both maximum flows in $D$, then $\mathcal{B}_{f_1} = \mathcal{B}_{f_2}$.
\end{lem}

\begin{proof}
We only need to show that all strongly connected components in $D_{f_1}$ correspond to those in $D_{f_2}$, differing only in edge direction.
Since $f_2 - f_1$ forms a circulation in $D_{f_1}$, without loss of generality, assume $f_2 = f_1 + \delta 1_K$, where $K$ is a directed cycle in $D_{f_1}$ and $\delta > 0$.
In this case, the strongly connected component containing $K$ remains strongly connected in $D_{f_2}$, and all other strongly connected components remain the same as in $D_{f_1}$.
\end{proof}

Thus, we can denote $\mathcal{B}_f$ simply as $\mathcal{B}$.
Let $\alpha(e)$ represent the set of endpoints of edge $e$, we have the following theorem:

\begin{thm} \label{Thm:net-dec}
Given $f$ as the maximum flow in $D$, we have:
\begin{enumerate}
\item $A = \{e \in E \mid f(e) > 0, \alpha(e) \setminus B \neq \emptyset, \forall B \in \mathcal{B}\}$;
\item $C = \{e \in E \mid \exists B \in \mathcal{B}, \alpha(e) \subseteq B\}$;
\item $R = \{e \in E \mid f(e) = 0, \alpha(e) \setminus B \neq \emptyset, \forall B \in \mathcal{B}\}$.
\end{enumerate}
\end{thm}

\begin{proof}
We only need to prove the second statement, after which other statements will naturally follow.
Let $e = (u, v) \in E$, and suppose there exists $B \in \mathcal{B}$ such that $\alpha(e) \subseteq B$.

If $e \in E_f$, it follows that $f(e) < c(e)$.
Since $D_f[B]$ is strongly connected, there exists a directed path from $v$ to $u$ in $D_f$.
By augmenting the flow along this cycle formed by the path and $e$, we obtain a new maximum flow $f'$, which satisfies $f'(e) < c(e)$.
Since $f'(e) > f(e) \geq 0$, it follows that $e \in C$.

If $e \notin E_f$, it follows that $f(e) = c(e)$.
Again, since $D_f[B]$ is strongly connected, there exists a directed path from $u$ to $v$ in $D_f[B]$.
Augmenting the flow along the cycle formed by this path and $(v, u)$ yields a new maximum flow $f'$ that satisfies $f'(e)>0$.
Since $f'(e)<f(e)=c(e)$, it follows that $e \in C$.

Conversely, suppose $e \in C$.
By definition, there exists a maximum flow $f'$ in $D$ such that $0 < f'(e) < c(e)$.
This implies that $u$ and $v$ are in the same strongly connected component of $D_{f'}$, meaning there exists a $B \in \mathcal{B}$ such that $\alpha(e) \subseteq B$.
\end{proof}

From this theorem, we can infer that dummy I edges in the network form strongly connected components in $D_f$.
The remaining edges consist of essential edges and dummy II edges.
Essential edges are necessary components of the maximum flow, while dummy II edges play various roles in the network.

In addition to classifying the edges in network $D$, the vertex sets in $\mathcal{B}$ can also be categorized based on their relationship with the maximum flow:

\begin{defn} \label{Defn:vertex-set}
Given $B \in \mathcal{B}$, we classify $B$ as follows:
\begin{enumerate}
    \item If $s \in B$, then $B$ is called a \textit{start vertex set};
    \item If $t \in B$, then $B$ is called an \textit{end vertex set};
    \item If $s, t \notin B$ and the number of essential edges entering to $B$ is greater than one, $B$ is called a \textit{transfer vertex set};
    \item If $s, t \notin B$ and exactly one essential edge enters to $B$, $B$ is called a \textit{direct vertex set};
    \item If $s, t \notin B$ and no essential edges enters to $B$, $B$ is called a \textit{removable vertex set}.
\end{enumerate}
\end{defn}

\begin{ex} \label{Ex:net}

In the following network, the capacities of all edges are equal to one.
The colors red, black, and green represent the essential edges, the dummy I edges, and the dummy II edges, respectively.
The corresponding types of vertex sets are labeled in the figure \ref{Fig:net}.

\begin{figure}[h]
    \centering
    \includegraphics[width=0.6\textwidth]{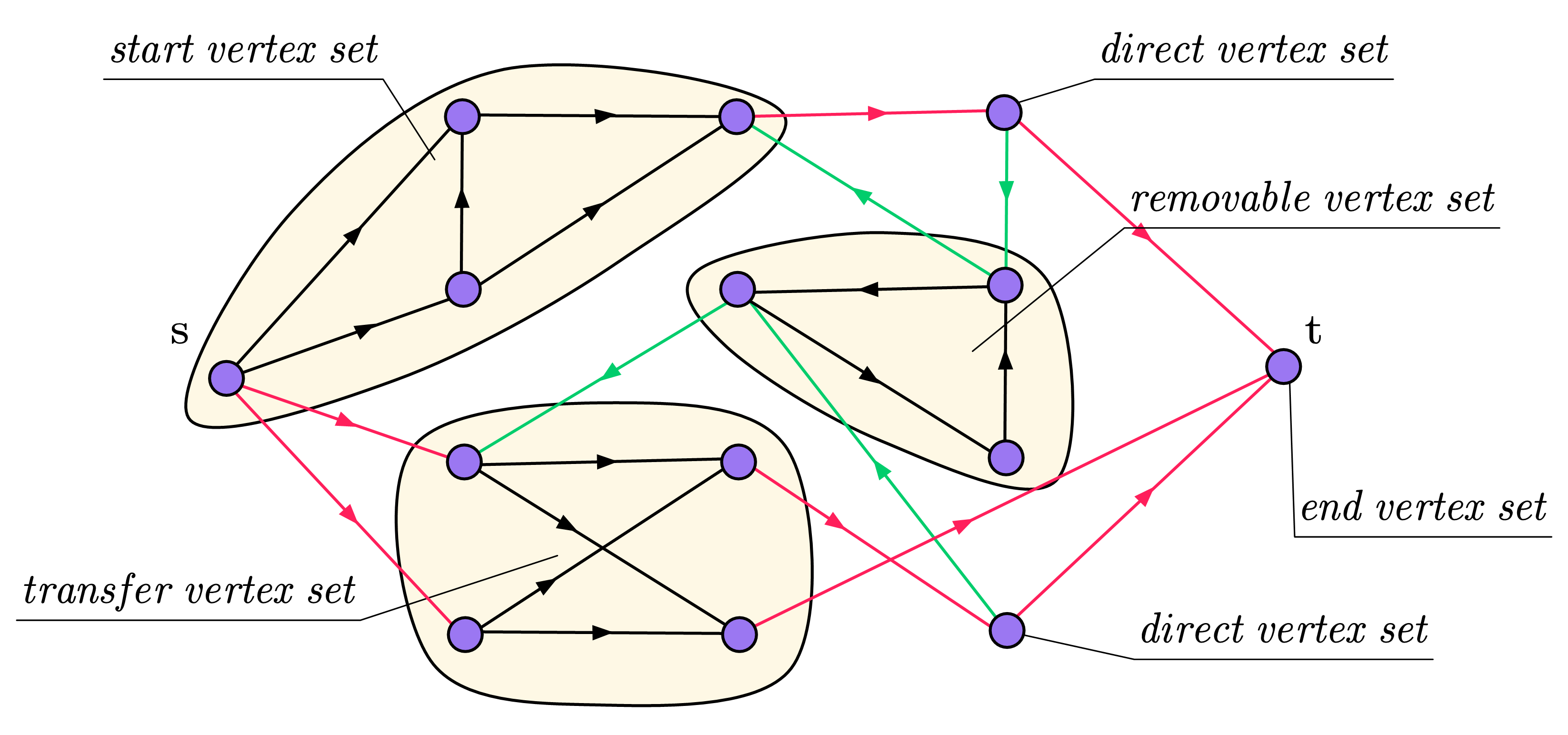} % 插入图片
    \caption{Three types of edges and five types of vertex sets in the network.}
    \label{Fig:net}
\end{figure}
\end{ex}

We conclude this section with the following simple properties.

\begin{prop} \label{Prop:remove}
Given $B \in \mathcal{B}$, the following statements are equivalent:
\begin{enumerate}
\item $B$ is a removable vertex set;
\item The edges in $D[B]$ can only form circulations in the maximum flow.
\end{enumerate}
\end{prop}

\begin{proof}
Proof omitted.
\end{proof}

When considering only the structure of paths in the network, the removable vertex sets can be iteratively eliminated to simplify the network.
However, in other contexts, such simplification may destroy the original connectivity structure of the network, hence the removable vertex sets are not truly “removable.”

Next, we assume that any dummy II edge lies on some $s$-$t$ path in $D$ (otherwise, it can be removed directly).
We can prove the following simple property.

\begin{prop} \label{Prop:start-end}
Let $e=(u,v)\in E$, the following facts holds:
\begin{enumerate}
\item If $e \in R$, then $u$ must not belong to the start vertex set, and $v$ must not belong to the end vertex set;
\item If $e \in A$, then $v$ must not belong to the start vertex set, and $u$ must not belong to the end vertex set.
\end{enumerate}

\end{prop}

\begin{proof}
We only prove that $u$ must not belong to the start vertex set in the first statement.
Otherwise, $u,v$ must belong to the start vertex set.
According to Theorem \ref{Thm:net-dec}, $e \in C$, which is a contradiction.
\end{proof}

\section{Decomposition of the Network and Minimum Cut} \label{Sec:dec-cut}

In this section, we study the distribution of minimum cuts in networks.
Picard and Queyranne \cite{Picard1980structure} provided a construction for any minimum cut in the network.

\begin{thm}[Picard-Queyranne \cite{Picard1980structure}] \label{Thm:PQ}
Let $f$ be the maximum flow of $D$.
Then the $s-t$ cut $\delta^+(S)$ is a minimum cut in $D$ if and only if for any $u \in S$, if $(u,v) \in E_f$, then $v \in S$.
\end{thm}

\begin{proof}
For any flow $f'$ on $D$ and the $s-t$ cut $\delta^+(S')$, the following inequality holds:
$$
c(\delta^+(S')) \geq f'(\delta^+(S')) - f'(\delta^-(S'))
$$
According to the max-flow min-cut theorem, equality in the above inequality holds if and only if $f'$ and $\delta^+(S')$ are the maximum flow and minimum cut of $D$, respectively.

Let $\delta^+(S)$ be the minimum cut of $D$.
There exists $(u,v) \in E_f$ such that $u \in S$ and $v \in V \setminus S$.
Thus, for any $(u,v) \in \delta^+(S)$, we have $f(u,v) = c(u,v)$ or $f(v,u) = 0$, meaning $(u,v) \notin E_f$, which is a contradiction.

On the other hand, assume $\delta^+(S)$ is an $s-t$ cut, and if $u \in S$ and $(u,v) \in E_f$, then $v \in S$.
Thus, for any $(u,v) \in \delta^+(S)$, we have $f(u,v) = c(u,v)$ or $f(v,u) = 0$.
Hence, the above inequality achieves equality, meaning $\delta^+(S)$ is a minimum cut.
\end{proof}

\begin{cor} \label{Cor:dec-cut}
Then the $s-t$ cut $\delta^+(S)$ is a minimum cut in $D$ if and only if for any $u \in S$, if $(u,v)\in E\setminus A$ or $(v,u)\in A$, then $v\in S$.
\end{cor}

\begin{proof}
Notice that $(u,v) \in E_f$ if and only if $(u,v)\in E\setminus A$ or $(v,u)\in A$.
\end{proof}

Let $\mathcal{C}$ denote the set of all minimum cuts in $D$.
We now study the distribution of the minimum cut in the network with respect to the sets $A$, $C$ and $R$.

\begin{lem} \label{Lem:A-cut}
$A = \bigcup_{\delta^+(S) \in \mathcal{C}} \delta^+(S)$.
\end{lem}

\begin{proof}
Let $e = (u,v) \in A$ and $f$ be the maximum flow of $D$.
The set $S$ consists of $s$, $u$, and all other vertices reachable from these two vertices in $D_f$.
By the Theorem \ref{Thm:PQ}, $\delta^+(S)$ is a minimum cut in $D$.

Moreover, $v$ must not be in $S$, otherwise there would be a directed path from $s$ or $u$ to $v$ in $D_f$.
If there is a directed path from $u$ to $v$ in $D_f$, it means that $v$ would be in the same strongly connected component with $u$ in $D_f$.
By Theorem \ref{Thm:net-dec}, we obtain $e \in C$, which is a contradiction.

If there is a directed path from $s$ to $v$ in $D_f$, it means that $v$ would be in the start vertex set by using Proposition \ref{Prop:start-end}.
According to Proposition \ref{Prop:start-end}, we obtain $e \in E\setminus A$, which is a contradiction.

Conversely, assume $e = (u,v) \in \delta^+(S) \in \mathcal{C}$ and $e \notin A$.
By definition \ref{Defn:edge}, there exists a maximum flow $f$ in $D$ such that $f(e) < c(e)$.
Thus, $e \in E_f$.
According to the Theorem \ref{Thm:PQ}, we obtain $v \in S$, which is a contradiction.
\end{proof}

\begin{lem} \label{Lem:C-cut}
If $\delta^+(S)$ is a minimum cut in $D$, then there exists $\mathcal{B}' \subseteq \mathcal{B}$ such that $S = \bigcup_{B \in \mathcal{B}'} B$.
\end{lem}

\begin{proof}
Using Theorem \ref{Thm:net-dec}, it is sufficient to prove that if $e \in C$, then $e \notin \delta^+(S) \cup \delta^-(S)$ for any $\delta^+(S) \in \mathcal{C}$.
By Corollary \ref{Cor:dec-cut}, we know that $e \notin \delta^-(S)$.
Additionally, Lemma \ref{Lem:A-cut} implies that $e \notin \delta^+(S)$.
Therefore, we conclude that $e \notin \delta^+(S) \cup \delta^-(S)$.
\end{proof}

Next, we analyze the relationship between dummy II edges and minimum cuts.

\begin{defn}
Let $u,v \in V$.
If there exists a path $P_{uv}$ in $D$ from $u$ to $v$ composed solely of dummy edges and containing at least one dummy II edge, then $P_{uv}$ is called a \textit{jump} from $u$ to $v$.
\end{defn}

Evidently, a jump must contain at least one dummy II edge, and it has the following relationship with minimum cuts.

\begin{prop} \label{Prop:jump}
Let $P_{uv}$ be a subpath in $D$ consisting of dummy edges only.
Then $P_{uv}$ is a jump if and only if there exists a minimum cut $\delta^+(S)$ such that $u \in V \setminus S$ and $v \in S$.
\end{prop}

\begin{proof}
Assume $P_{uv}$ is a jump, and let $(p,q) \in R$ be a dummy II edge on $P_{uv}$.
Let $f$ be the maximum flow of $D$.
Let $S$ be the set consisting of $s$, $q$, $v$, and all other vertices reachable from these two vertices in $D_f$.
By Theorem \ref{Thm:PQ}, $\delta^+(S)$ is a minimum cut in $D$.
Since there is no directed path from $v$ to $u$ or $p$ in $D_f$ (otherwise, these two vertices would be in the same strongly connected component in $D_f$), it follows that $u, p \in V \setminus S$.

On the other hand, suppose that there exists a minimum cut $\delta^+(S)$ such that $u \in V \setminus S$ and $v \in S$.
If $P_{uv}$ consists of dummy I edges only, according to Corollary \ref{Lem:C-cut}, we would have $u, v \in S$ or $u, v \in V \setminus S$, which is a contradiction.
\end{proof}

It is worth noting that even if both ends of a jump are in the transfer vertex set or direct vertex set, it does not guarantee the existence of other essential edges and dummy I edges that together form an $s-t$ path, as illustrated in Example \ref{Ex:net}.

Since dummy II edges are minimal jumps, we have the following corollary.

\begin{cor} \label{Cor:R-cut}
$R = \bigcup_{\delta^+(S) \in \mathcal{C}} \delta^-(S)$.
\end{cor}

\begin{proof}
On the one hand, we assume that $e\in R$.
We can proof that $e\in \delta^-(S)$ for some $\delta^+(S) \in \mathcal{C}$ directly by taking jump as $e$ in Proposition \ref{Prop:jump}.

On the other hand, if $e\in \delta^-(S)$ for some $\delta^+(S) \in \mathcal{C}$, it follows that $e\in R$ by using Lemma \ref{Lem:A-cut} and Lemma \ref{Lem:C-cut}.
\end{proof}

\section{Decomposition of Networks and Potential Functions} \label{Sec:dec-pot}

Imagine water flowing down a slope; during this process, the gravitational potential energy of the water gradually transforms into kinetic energy necessary for flow.
The potential function in the network describes the "potential" energy of the network flow at the vertices, defined as follows:

\begin{defn}
Let $\pi: V \rightarrow [0, 1]$ be a function on the vertices.
If $\pi$ satisfies:

\begin{enumerate}
\item $\pi(s) = 1$, $\pi(t) = 0$,
\item If $(u,v) \in A$, then $\pi(u) \geq \pi(v)$,
\item If $(u,v) \in E \setminus A$, then $\pi(u) \leq \pi(v)$,
\end{enumerate}

then $\pi$ is called a \textit{potential function} on $D$.
\end{defn}

The relationship between potential functions and minimum cuts is as follows:

\begin{lem} \label{Lem:cut-pot}
Let $S \subseteq V$.
Then $\delta^+(S)$ is a minimum cut of $D$ if and only if $\chi(S)$ is a potential function.
\end{lem}

\begin{proof}
Assume $\delta^+(S)$ is a minimum cut of $D$, and let $e = (u,v) \in E$ be an edge in $D$.
If $e \in A$, by Lemma \ref{Lem:A-cut}, we have $\chi(S)(u) \geq \chi(S)(v)$.
If $e \in C$, by Corollary \ref{Lem:C-cut}, it follows that $\chi(S)(u) = \chi(S)(v)$.
If $e \in R$, by Corollary \ref{Cor:R-cut}, we have $\chi(S)(u) \leq \chi(S)(v)$.
Thus, we proof that $\chi(S)$ is a potential function.

Conversely, assume $\chi(S)$ is a potential function.
We will prove that $\delta^+(S)$ is a minimum cut of $D$.
Since $\chi(S)(s) = 1$ and $\chi(S)(t) = 0$, we conclude that $s \in S$ and $t \in V \setminus S$.
Let $f$ be the maximum flow of $D$, and assume $u \in S$ and $(u,v) \in E_f$.
According to the Theorem \ref{Thm:PQ}, we need only show that $v \in S$.

If $(u,v) \in A$, then $(v,u) \in E_f$.
By the definition of the potential function, we have $1 \geq \chi(S)(v) \geq \chi(S)(u) = 1$, which implies $v \in S$.
If $(u,v) \in E \setminus A$, then $(u,v) \in E_f$.
By the definition of the potential function, we have $1 \geq \chi(S)(v) \geq \chi(S)(u) = 1$, which also implies $v \in S$.
\end{proof}

This conclusion can be extended to the case of convex hulls.
Let the set of all potential functions on $D$ be denoted as $\Pi$, and define
$$
\text{conv}(\mathcal{C}) = \text{conv}\{\chi(\delta^+(S)) \mid \delta^+(S) \in \mathcal{C}\}
$$
as the convex hull of all minimum cuts, where $\chi(\delta^+(S))$ is the indicator vector of $\delta^+(S)$.
For any edge $e = (u,v) \in E$, let $\text{diff}(\pi,e) = \pi(u) - \pi(v)$ denote the \textit{potential difference} of $e$ under $\pi$.
Given a potential function $\pi$, define $\text{diff}^*(\pi) \in \mathbb{R}^{|E|}$ as follows:
$$
\text{diff}^*(\pi)(e) = 
\begin{cases}
\text{diff}(\pi,e), & e \in A, \\
0, & e \in E \setminus A.
\end{cases}
$$

The following lemma illustrates the one-to-one correspondence between potential functions and fractional minimum cuts:

\begin{lem}
$\text{conv}(\mathcal{C}) = \{\text{diff}^*(\pi) \mid \pi \in \Pi\}$
\end{lem}

\begin{proof}
Let $\delta^+(S_1),\ldots,\delta^+(S_k)$ be all minimum cuts of $D$.
For any $x \in \text{conv}(\mathcal{C})$, since $\text{conv}(\mathcal{C})$ is a convex set, we can express $x$ as:
$$
x = \lambda_1 \chi(\delta^+(S_1)) + \ldots + \lambda_k \chi(\delta^+(S_k))
$$
where $0 \leq \lambda_i \leq 1$ and $\sum_{i=1}^k \lambda_i = 1$.
Define $\pi = \sum_{i=1}^k \lambda_i \chi(S_i)$.
By Lemma \ref{Lem:cut-pot}, it is easy to verify that $\pi$ is a potential function.
Furthermore, if $e \in E \setminus A$, then $\text{diff}^*(\pi)(e) = 0$.
If $e \in A$, we have
$$
\text{diff}^*(\pi)(e) = \sum_{i=1}^k \lambda_i \text{diff}(\chi(S_i),e) = \sum_{i=1}^k \lambda_i \chi(\delta^+(S_i))(e) = x(e),
$$
thus $x = \text{diff}^*(\pi)$.

Conversely, let $\pi \in \Pi$.
Define $\phi_0 = 0$, $S_{k} = \{v \in V \mid \pi(v) > \phi_{k-1}\}$, and $\phi_k = \min_{v \in S_k} \{\pi(v)\}$ for $k \geq 1$.
It is easy to verify that $\chi(S_k)$ is a potential function and that $S_{k+1} \subset S_k$, with $\phi_{k+1} < \phi_k$.
Define $S_m = \{v \in V \mid \pi(v) = 1\}$, leading to
$$
\pi = \phi_1 \chi(S_1) + (\phi_2 - \phi_1) \chi(S_2) + \ldots + (\phi_m - \phi_{m-1}) \chi(S_m).
$$
Thus, by Lemma \ref{Lem:cut-pot},
$$
\begin{aligned}
\text{diff}^*(\pi) &= \phi_1 \text{diff}^*(\chi(S_1)) + (\phi_2 - \phi_1) \text{diff}^*(\chi(S_2)) + \ldots + (\phi_m - \phi_{m-1}) \text{diff}^*(\chi(S_m)) \\
&= \phi_1 \chi(\delta^+(S_1)) + (\phi_2 - \phi_1) \chi(\delta^+(S_2)) + \ldots + (\phi_m - \phi_{m-1}) \chi(\delta^+(S_m)).
\end{aligned}
$$
Here, $\phi_1 + (\phi_2 - \phi_1) + \ldots + (\phi_m - \phi_{m-1}) = \phi_m = 1$.
Therefore, $\text{diff}^*(\pi) \in \text{conv}(\mathcal{C})$.
\end{proof}

The above lemma indicates a correspondence between potential functions and fractional minimum cuts.
The relationships between potential functions and sets $A, C, R$ are as follows:

\begin{lem}
The following facts hold:

\begin{enumerate}
\item $A = \{e \in E \mid \exists \pi \in \Pi, ~ \text{diff}(\pi,e) > 0\}$;
\item $C = \{e \in E \mid \forall \pi \in \Pi, ~ \text{diff}(\pi,e) = 0\}$;
\item $R = \{e \in E \mid \exists \pi \in \Pi, ~ \text{diff}(\pi,e) < 0\}$.
\end{enumerate}
\end{lem}

\begin{proof}
The statements follow immediately from Lemma \ref{Lem:A-cut}, Corollary \ref{Lem:C-cut}, and Corollary \ref{Cor:R-cut}.
\end{proof}

From the second statement of this theorem, we obtain the following corollary:

\begin{cor}
Given $B \in \mathcal{B}$ and $\pi \in \Pi$, for any $u,v \in B$, we have $\pi(u) = \pi(v)$.
\end{cor}

Finally, we illustrate that potential functions and the corresponding fractional minimum cuts appear in pairs in the linear program (\ref{LP:re-min-cut}):

\begin{lem}
$(y, \pi) \in \mathbb{R}^{|E| \times |V|}$ is an optimal solution of linear program (\ref{LP:re-min-cut}) if and only if $\pi \in \Pi$, $y = \text{diff}^*(\pi)$.
\end{lem}

\begin{proof}
Let $\pi \in \Pi$ and $y = \text{diff}^*(\pi)$, and let $e = (u,v)$ be an edge in $D$.
If $e \in A$, then $\pi(u) - \pi(v) - y(u,v) = 0$.
If $e \in E \setminus A$, then $y(u,v) = 0$, so $\pi(u) - \pi(v) - y(u,v) = 0$.
Since $\sum_{e \in E} c(e) y(e)$ is the value of the maximum flow, by the strong duality theorem, $(y, \pi)$ is the optimal solution of the linear program.

Conversely, let $(y, \pi)$ be the optimal solution of the linear program, and let $e = (u,v)$ be an edge in $D$.
By the complementary slackness condition, if $e \in A$, then $\pi(u) - \pi(v) = y(u,v) \geq 0$, which implies $\pi(u) \geq \pi(v)$.
If $e \in C$, then $\pi(u) - \pi(v) = y(u,v) = 0$.
If $e \in R$, then $\pi(u) - \pi(v) - y(u,v) = \pi(u) - \pi(v) \geq 0$, which implies $\pi(v) \geq \pi(u)$.
Furthermore, for any $v \in V$, we have $1 = \pi(s) \geq \pi(v) \geq \pi(t) = 0$.
Therefore, $\pi \in \Pi$ and $y = \text{diff}^*(\pi)$.
\end{proof}

\bibliographystyle{amsalpha} 
\bibliography{reference} 

\Addresses

\end{document}